\documentclass[11pt]{article}
\usepackage{amsmath}
\usepackage{amsthm}
\usepackage{amssymb,amsfonts}
\usepackage{todonotes}
\usepackage{nicefrac}
\usepackage{epsfig}
\usepackage{stmaryrd}
\usepackage{setspace}
\usepackage{tikz}
\usepackage{enumerate}
\usepackage[all]{xypic}
\usepackage{bbm,ifpdf,tikz,mathtools}
\usepackage{enumitem,multicol}

\usepackage[plainpages]{hyperref}
\usepackage{cleveref}

\usetikzlibrary{cd,positioning,matrix,arrows,decorations.pathmorphing,calc}
\usetikzlibrary{decorations.markings}
\tikzstyle{point}=[draw, black, fill,shape=circle, minimum size=4pt, inner sep=0pt]
\usepackage{caption, subcaption}
\tikzset{squiggly/.style={decorate, decoration=snake}}

\newtheorem{theorem}{Theorem}[section]

\newtheorem{lemma}[theorem]{Lemma}
\newtheorem{proposition}[theorem]{Proposition}
\newtheorem{corollary}[theorem]{Corollary}

\newtheorem*{main-result}{Main-Result}

\theoremstyle{definition}

\newtheorem{example}[theorem]{Example}

\newtheorem{remark}[theorem]{Remark}

\crefname{claim}{Claim}{Claims}

\renewcommand{\dim}{\operatorname{dim}}

\newcommand{\excise}[1]{}

\newcommand{\Q}{{\mathbb{Q}}}
\newcommand{\R}{{\mathbb{R}}}

\renewcommand{\a}{\alpha}

\newcommand{\eps}{\epsilon}

\newcommand{\F}{\mathbb{F}}

\renewcommand{\iff}{\Leftrightarrow}

\newcommand{\A}{\mathcal{A}}

\newcommand{\cA}{\mathcal{A}}
\newcommand{\cB}{\mathcal{B}}

\newcommand{\cC}{\mathcal{C}}

\newcommand{\cS}{S}

\def\checkmark{\tikz\fill[scale=0.4](0,.35) -- (.25,0) -- (1,.7) -- (.25,.15) -- cycle;}

\renewcommand{\and}{\qquad\text{and}\qquad}
\newcommand{\VG}{\operatorname{VG}}

\renewcommand{\P}{\mathbb{P}}

\newcommand{\rk}{\operatorname{rk}}
\newcommand{\gr}{\operatorname{gr}}

\newcommand{\Cor}{\operatorname{C}}

\newcommand{\abs}[1]{\left|#1\right|}               
\newcommand{\set}[1]{\left\{#1\right\}}             
\newcommand{\angl}[1]{\left<#1\right>}              

\newcommand{\drawline}[2]
{
\draw ($ #1 ! -0.5 ! #2 $) -- ($ #1 ! 1.5 ! #2 $);
}

\usepackage{todonotes}

\usepackage{todonotes}

\begin{document}
\spacing{1.2}
\noindent{\Large\bf Cleanliness and the Varchenko--Gelfand algebra}\\

\noindent{\bf Graham Denham}\footnote{Partially supported by a grant from NSERC of Canada}\\
Department of Mathematics, University of Western Ontario, London, ON
\vspace{.1in}

\noindent{\bf Galen Dorpalen-Barry}\footnote{Supported by NSF grant DMS-2039316.}\\
Department of Mathematics, Texas A\&M University, College Station, TX
\vspace{.1in}

\noindent{\bf Nicholas Proudfoot}\footnote{Supported by NSF grants DMS-2039316 and DMS-2344861.}\\
Department of Mathematics, University of Oregon, Eugene, OR\\

{\small
\begin{quote}
\noindent {\em Abstract.} 
A central question in the theory of hyperplane arrangements is when the complement of a complex
arrangement is $K(\pi,1)$.  
Barkley and Speyer introduced a class of real arrangements that are called ``clean,'' and
Yoshinaga proved that every real arrangement whose complexification is $K(\pi,1)$ is clean.
We show that cleanliness is equivalent to a natural statement about the Varchenko--Gelfand ring, which in practice allows
for fast calculation.  We conclude with an investigation of the relationships between various properties of arrangements, including cleanliness and the $K(\pi,1)$ property.
\end{quote} }

\section{Introduction}
It is a long-standing open problem to determine which complex hyperplane arrangement complements are $K(\pi,1)$, meaning that their higher homotopy groups
vanish.  In the case where the hyperplane arrangement is the complexification of a real hyperplane arrangement,
the homotopy type of the complement is determined by the associated oriented matroid \cite[Theorem 1]{salvetti}, and therefore the $K(\pi,1)$ problem
must have a combinatorial answer.  See \cite{Falk-Randell2,FR86,yoshinaga-survey} for surveys and partial results.

Recently, Yoshinaga proved that, if $\cA$ is a real hyperplane arrangement and the complement of the complexification of $\cA$
is $K(\pi,1)$, then $\cA$ is {\bf clean} in the sense of Barkley and Speyer \cite{Barkley-Speyer,yoshinaga-spheres}. 
In the first part of this paper, we give an algebraic reformulation of cleanliness,
which we now describe.

Let $\cA$ be a finite set of hyperplanes in a real vector space $V$, and let $\F$ be any field.  The {\bf Varchenko--Gelfand algebra}
$\VG(\cA,\F)$ is by definition the ring of locally constant $\F$-valued functions on the complement of the union of hyperplanes.
This is a boring ring (it is isomorphic to a direct sum of one copy of $\F$ for each chamber), but it admits an interesting presentation
whose generators are the {\bf Heaviside functions}: there are two such functions for each hyperplane, taking the value 1 on one side
of the hyperplane and 0 on the other side.  The ring $\VG(\cA,\F)$ is filtered, with the $p^\text{th}$ filtered piece consisting as functions
that can be expressed as polynomials of degree at most $p$ in the Heaviside functions.  The associated graded algebra,
which is also called the {\bf Cordovil algebra},
is isomorphic to the cohomology ring of the complement of the union of the subspaces $H\otimes\R^3\subset V\otimes\R^3$ \cite{moseley,DBPW}.
We say that the Varchenko--Gelfand algebra is {\bf quadratic} if all relations among the Heaviside functions are generated by those of degree
at most 2.  Similarly, we say that the Cordovil algebra is quadratic if all relations among the corresponding generators are generated by those of
degree 2.  
Our main results (Theorem \ref{main} and Corollary \ref{cor:cordovil}) say that the following implications hold:
$$\text{$\Cor(\cA,\F)$ is quadratic $\Longrightarrow$ $\VG(\cA,\F)$ is quadratic $\Longleftrightarrow$ $\cA$ is clean.}$$

\begin{remark}\label{rem:speed}
At first sight, cleanliness (which is formulated combinatorially) might seem easier to work with than the condition that 
$\VG(\cA,\F)$ is quadratic.  In fact, our experience is that the algebraic condition is much faster to check, since computers are very good
at using Gr\"obner bases to determine whether or not two ideals are equal.
For example, let $\cA$ be the arrangement whose normal vectors are given by the columns of the following matrix:
\[
\left(
\begin{array}{cccccccccccccccccccc}
3 & 3 & 3 & 3 & 3 & 9 & 7 & 5 & 7 & 2 & 0 & 0 & 6 & 3 & 4 & 8 & 6 & 2 & 9 & 5 \\
8 & 1 & 7 & 1 & 2 & 8 & 2 & 6 & 1 & 8 & 5 & 9 & 2 & 8 & 3 & 0 & 1 & 0 & 8 & 9 \\
1 & 9 & 1 & 9 & 5 & 2 & 5 & 9 & 3 & 7 & 7 & 3 & 6 & 6 & 4 & 0 & 9 & 1 & 5 & 9 \\
1 & 0 & 1 & 4 & 1 & 1 & 7 & 2 & 4 & 1 & 3 & 9 & 2 & 8 & 0 & 8 & 7 & 1 & 2 & 3 \\
\end{array}
\right)
\]
It took about 1.65 seconds for Macaulay2 to determine that the Varchenko--Gelfand ideal is not quadratic.
On the other hand, it took 3 days, 23 hours, 21 minutes, and 22 seconds for Sage to check cleanliness directly.\footnote{Our 
Sage implementation could admit many improvements; but, even with considerable effort, it is unlikely that we could beat the time of the 
easy Macaulay2 calculation.}
\end{remark}

Our primary motivation for Theorem \ref{main} is to be able to perform fast calculations, 
and in particular to probe the question of how close cleanliness is to the $K(\pi,1)$ property. 
The following examples illustrate the type of calculations that are made possible by our result.

\begin{example}
Let $\cA$ be the arrangement in $\mathbb{R}^6$ consisting of all hyperplanes of the from $x_i = x_j$ for $1\leq i<j\leq 6$, 
together with the hyperplanes $x_i + x_j = 0$ whenever $j - i$ is prime.
(This is an intentionally unmotivated condition that is meant to produce a somewhat random arrangement lying in between
the Coxeter arrangements of type $A_5$ and $D_6$.)
The ring $\VG(\cA,\Q)$ is not quadratic, and therefore $\cA$ is not $K(\pi,1)$.
This can be checked in Macaulay22 in about 30 seconds.
\end{example}

\begin{example}
  For $t\in\mathbb{R}$, let $\cA_t$ be the arrangement with hyperplanes
  \begin{align*}
  x_1 - x_2 & = 0,\;\; x_1 - x_3 = 0,\;\; x_2 - x_3 = 0,\;\; x_1 = 0,\;\; x_2 = 0,\;\; x_3 = 0, \\ x_1 - t\thinspace x_2 & = 0,\;\; x_1 - t \thinspace x_3 = 0,\;\; x_2 - t \thinspace x_3 = 0\,.
  \end{align*}
When $t\in\{-1,0,1\}$, these arrangements have quadratic Varchenko--Gelfand algebras, and are therefore clean.
Edelman--Reiner, however, show that these arrangements are \emph{not} $K(\pi,1)$ \cite[Theorem 2.1]{Edelman-Reiner3}.
\end{example}

\begin{example}
  Let $\cA$ be the {\bf bracelet arrangement} with hyperplanes
  \begin{align*}
  x_1 = 0,\;\; & x_2 = 0,\;\; x_3 = 0,\;\; x_1+x_4 = 0,\;\; x_2+x_4 = 0,\;\; x_3+x_4 = 0, \\ & x_1+x_2+x_4 = 0,\;\; x_1+x_3+x_4 = 0,\;\; x_2+x_3+x_4 = 0\,.
  \end{align*}
  This is the smallest known non-tame arrangement (see \cite{Ab25} for
  background).
Then $\VG(\cA;\Q)$ is quadratic, thus $\cA$ is clean.  
It is not known to the authors whether or not $\cA$ is $K(\pi,1)$.
Yoshinaga's theorem gives supporting evidence that it could be.
\end{example}

Section \ref{sec:relationships} is devoted to relating cleanliness to other algebraic, topological, and combinatorial conditions.  We define what it means
for a matroid to be {\bf chordal}, generalizing the notion of a chordal graph.  We then say that $\cA$ is chordal if its underlying matroid is chordal.
We prove that every real, chordal arrangement is clean (Theorem \ref{thm:chordal-means-yoshi}).   
We also provide a proof (communicated to us by Paul M\"ucksch) that every clean arrangement is {\bf formal}.
The converses to these two theorems are false (Example \ref{ex:x2-arr} and Remark \ref{formal converse}), but for graphical arrangements,
chordality, formality, cleanliness, and the $K(\pi,1)$ property are all equivalent (Corollary \ref{cor:yoshi-typeA}).
Finally, we provide a chart that illustrates the implications between various properties known to be related to the $K(\pi,1)$ property, including chordality, formality,
cleanliness, and more.

\subsection*{Acknowledgements}
The authors would like to thank Nick Addington, Grant Barkley, Mike Falk, Paul M\"ucksch, and Takuya Saito for their valuable contributions.

\section{Cleanliness and the Varchenko--Gelfand algebra}\label{sec:theory}
The main purpose of this section is to state and prove Theorem \ref{main} and Corollary \ref{cor:cordovil}.
Let $V$ be a real vector space of dimension $r$, $\cA$ a finite set of distinct hyperplanes in $V$ intersecting only at the origin (a {central, essential} arrangement), and
$M_d(\cA)$ the complement of the union of the subspaces $H\otimes \R^d\subset V\otimes \R^d$ for all $H\in\cA$.
In particular, $M_1(\cA)$ is the complement of $\cA$ (a union of contractible chambers), 
$M_2(\cA)$ is the complement of the complexification of $\cA$, 
and $M_3(\cA)$ is a space with cohomology ring isomorphic to the Cordovil algebra.
Let $\cC(\cA)$ be the set of chambers of $\cA$, that is, the connected components of $M_1(\cA)$.

\subsection{Cleanliness}
We begin by choosing coorientations of each element of $\cA$.  That is, for each $H\in\cA$, we write $H^+$ to denote 
one of the two connected components of $V\setminus H$, and $H^-$ to denote the other one.
To match the conventions in \cite{Barkley-Speyer}, we choose our coorientations in such a way so that the intersection of
all of the positive half-spaces is nonempty.
For any sign vector $\eps\in\{\pm\}^\cA$ and any subset $\cS\subset\cA$, let
\[
H_\cS^\eps := \bigcap_{H\in\cS}H^{\eps_H}.
\]
We say that $\eps$ is {\bf \boldmath{$k$}-consistent} if, for any subset $\cS$ of cardinality at most $k+1$, we have $H_\cS^\eps\neq\emptyset$.
Let $\Sigma_k = \Sigma_k(\cA)$ denote the set of $k$-consistent sign vectors, and let $\sigma_k := \abs{\Sigma_k}$.  
All sign vectors lie in $\Sigma_1$, and $\Sigma_r$ is naturally in bijection with $\cC(\cA)$, hence we have
\[
2^{\abs{\cA}} = \sigma_1 \geq \sigma_2 \geq \cdots \geq \sigma_{r - 1} \geq  \sigma_{r} = \abs{\cC(\cA)}.
\]
We say that $\cA$ is {\bf clean} if $\sigma_2 = \sigma_r$.

\begin{remark}
Our assumption that the intersection of all of the positive half spaces are nonempty implies that, if $\eps$ is a sign vector and
$S\subset\cA$ is a set of cardinality 3 with $H_S^\eps=\emptyset$, then the restriction of $\eps$ to $S$ 
either takes the value $+$ twice and $-$ once, or vice-versa.
In the terminology of \cite[Section 2.1]{Barkley-Speyer}, the sign vector $\eps$ is {\bf closed} if there does not exist 
such an $S$ such that the restriction of $\eps$ to $S$ takes the value $+$ twice, and it is {\bf coclosed} if there does not exist 
such an $S$ such that the restriction of $\eps$ to $S$ takes the value $-$ twice.
The sign vector $\eps$ is {\bf biclosed} if it is both closed and coclosed, which means that there is no set $S$ of cardinality 3 with $H_S^\eps=\emptyset$,
or equivalently that $\eps\in\Sigma_2(\cA)$.
Finally, $\eps$ is {\bf separable} if $H_{\!\cA}^\eps\neq\emptyset$, or equivalently if $\eps\in\Sigma_r(\cA)$.  Thus cleanliness is precisely the statement that every biclosed sign
vector is separable.
\end{remark}

Our interest in clean arrangements comes from the following result \cite[Theorem 5.1(2)]{yoshinaga-spheres}.

\begin{theorem}\label{thm:yoshinaga}
If $M_2(\cA)$ is $K(\pi,1)$, then $\cA$ is clean.
\end{theorem}

Note that the converse to \Cref{thm:yoshinaga} is false \cite[Example 5.5]{yoshinaga-spheres}.

\subsection{The Varchenko--Gelfand algebra}
Fix a field $\F$.  The {\bf Varchenko--Gelfand algebra} $\VG(\cA,\F)$ is defined to be the ring of locally constant functions from $M_1(\cA)$ to $\F$.
This is simply a direct sum of $\sigma_d$ copies of $\F$, one for each chamber of $\cA$.
However, this boring ring has an interesting presentation, which we now describe.

Consider the commutative $\F$-algebra
\[
R := \F[e_H^+\mid H\in\cA]\big{/}\angl{(e_H^+)^2 - e_H^+\mid H\in\cA}
\]
generated by one idempotent class for each hyperplane.
We will also define $e_H^- := 1 - e_H^+\in R$, so that $e_H^- e_H^+ = 0$ and $e_H^- + e_H^+ = 1$.  
Given a sign vector $\eps\in\{\pm\}^\cA$ and a subset $\cS\subset\cA$, let
\[
f_\cS^\eps := \prod_{H\in\cS} e_H^{\eps_H}\in R.
\]
Then $\{f_\cA^\eps\mid \text{$\eps$ a sign vector}\}$
is an additive basis of pairwise orthogonal idempotents in $R$.

There is a surjective $\F$-algebra homomorphism
$\varphi:R\to\VG(\cA,\F)$ taking $e_H^{\pm}$ to the {\bf Heaviside function} that takes the value 1 on $H^{\pm}$ and 0 on $H^{\mp}$.
Let us try to understand the kernel of $\varphi$.
$H_\cS^\eps = \emptyset$, then $f_\cS^{\eps}$ lies in the kernel of $\varphi$.
If $-\eps$ is the opposite sign vector, then $H_\cS^{-\eps} = - H_\cS^\eps = \emptyset$, so $f_\cS^{-\eps}$
also lies in the kernel of $\varphi$.
Let $g_{\cS}^\eps := f_{\cS}^\eps - f_{\cS}^{-\eps}$,
which has the property that
\[
f_{\cS}^\eps = e_H^{\eps_H} g_{\cS}^\eps \and f_{\cS}^{-\eps} = -e_H^{-\eps_H} g_{\cS}^\eps
\]
for any $H\in\cS$.
The following theorem of Varchenko and Gelfand \cite[Theorem 6]{VarchenkoGelfand} says that these classes generate the kernel.

\begin{theorem}\label{thm:VG}
The kernel of $\varphi$ is generated by the classes $g_\cS^{\eps}$ for all $\eps$ and $\cS$ such that $H_\cS^\eps = \emptyset$.
\end{theorem}

In order to relate cleanliness to the Varchenko--Gelfand ring, we introduce a family of smaller ideals that sit inside the kernel of $\varphi$.
For any $k$, we define the {\bf \boldmath{$k^\text{th}$} intermediate Varchenko--Gelfand ideal}
\[
I_k := \angl{g_{\cS}^\eps \mid \text{$H_\cS^\eps = \emptyset$ and $\abs{\cS}\leq k+1$}}\subset R,
\]
and the {\bf \boldmath{$k^\text{th}$} intermediate Varchenko--Gelfand algebra} $\VG_k(\cA,\F) := R/I_k$. 
We have containments
\[
0 = I_1 \subset I_2 \subset\cdots\subset I_{r-1}\subset I_r = \ker(\varphi),
\]
along with quotients
\[
R = \VG_1(\cA,\F)\twoheadrightarrow \VG_2(\cA,\F)\twoheadrightarrow\cdots\twoheadrightarrow\VG_{r-1}(\cA,\F)\twoheadrightarrow\VG_{r}(\cA,\F)=\VG(\cA,\F).
\]

The following lemma gives an additive basis for the ideal $I_k$.

\begin{lemma}\label{basis}
We have $I_k = \F\{f_\cA^\eps\mid \eps\notin\Sigma_k\}$.
\end{lemma}

\begin{proof}
If $\eps\notin\Sigma_k$, then there is a subset $\cS\subset \cA$ of cardinality $k+1$ such that $H_\cS^\eps=\emptyset$, and
therefore $g_\cS^\eps\in I_k$.  We have already observed that $f_\cS^\eps$ is a multiple of $g_\cS^\eps$, and $f_\cA^\eps$
is by definition a multiple of $f_\cS^\eps$, so we also have $f_\cA^\eps\in I_k$.  This proves that $\F\{f_\cA^\eps\mid \eps\notin\Sigma_k\}\subset I_k$.

Next, we prove the opposite inclusion.  Since $\{f_\cA^\eps\mid \text{$\eps$ a sign vector}\}$
is an additive basis of pairwise orthogonal idempotents in $R$, $\F\{f_\cA^\eps\mid \eps\notin\Sigma_k\}$ is an ideal, and therefore it is sufficient to show that
the generators of $I_k$ are contained in $\F\{f_\cA^\eps\mid \eps\notin\Sigma_k\}$.

Let $\cS$ be a set of cardinality at most $k+1$ and $\delta$ a sign vector such that $H_\cS^\delta = \emptyset$.
We have
\[f_\cS^\delta = \sum_{\delta|_\cS = \eps|_\cS} f_\cA^\eps.
\]
For all $\eps$ such that $\delta|_\cS = \eps|_\cS$, we have $H_\cS^\eps = H_\cS^\delta = \emptyset$, and therefore $\eps\notin\Sigma_k$.
Thus we have established that $f_\cS^\delta\in \F\{f_\cA^\eps\mid \eps\notin\Sigma_k\}$.  By symmetry, we also have $f_\cS^{-\delta}\in\F\{f_\cA^\eps\mid \eps\notin\Sigma_k\}$,
and therefore $g_\cS^\delta = f_S^\delta - f_S^{-\delta} \in \F\{f_\cA^\eps\mid \eps\notin\Sigma_k\}$.  This completes the proof.
\end{proof}

\begin{theorem}\label{main}
For all $k$, $\sigma_k = \dim \VG_k(\cA,\F)$.  In particular, $\cA$ is clean if and only if $I_2 = I_r$.
\end{theorem}

\begin{proof}
By Lemma \ref{basis}, the set $\{f_\cA^\eps\mid\eps\in\Sigma_k\}\subset R$ descends to a basis for $\VG_k(\cA,\F)$.
\end{proof}

\subsection{The Cordovil algebra}
One reason for studying the Varchenko--Gelfand algebra is that it admits a natural filtration whose associated graded is of independent interest.
Consider the increasing filtration of $R$ whose degree $p$ piece consists of all classes that can be expressed as polynomials
of degree at most $p$ in the generators $e_H^\pm$, and let
\[
\bar R := \F[e_H\mid H\in\cA]\big{/}\langle e_H^2\mid H\in\cA\rangle
\]
be the associated graded algebra with respect to this filtration.  For any element $g\in R$, we write $\bar g\in \bar R$ to denote the {\bf symbol}
of $f$.  In concrete terms, this means that we express $f$ as a polynomial in the classes $e_H^+$, take the part of maximal degree, and replace
each $e_H^+$ with $e_H$.

For any ideal $I\subset R$, let $\bar I := \langle \bar g\mid g\in I\rangle$.
Our filtration of $R$ induces a filtration of $R/I$, and the associated graded algebra is isomorphic to $\bar R/\bar I$.
In particular, it induces a filtration of $\VG(\cA,\F) \cong R/I_r$, and the associated graded algebra 
\[
\Cor(\cA,\F) := \gr\VG(\cA,\F) \cong \bar R/\bar I_r
\]
is called the {\bf Cordovil algebra} (or sometimes the {\bf graded Varchenko--Gelfand algebra}) of $\cA$.
It follows from \cite[Theorem 7]{VarchenkoGelfand} that 
\[
\bar I_r = \big\langle \overline{g_{\cS}^\eps} \mid H_\cS^\eps = \emptyset\big\rangle.
\]

Just as in the filtered case, we can define intermediate versions of the Cordovil ideal.
For each $k$, we define the {\bf \boldmath{$k^\text{th}$} intermediate Cordovil ideal}
\[
J_k := \big\langle \overline{g_{\cS}^\eps} \mid \text{$H_\cS^\eps = \emptyset$ and $\abs{\cS}\leq k+1$}\big\rangle\subset \bar I_k.
\]
We have $J_1 = 0 = \bar I_1$ and $J_r = \bar I_r$, and $J_k \subset J_r$ is the sub-ideal generated by elements of degree at most $k$.  
In general, however,
the inclusion $J_k\subset \bar I_k$ can be proper.  That is, we have the following diagram of ideals:
\[
\begin{tikzpicture}[baseline=(current bounding box.center), node distance=0.25cm]

  \node (I0)    at (0,0)            {$0$};
  \node (s1)    [right=of I0]       {$=$};
  \node (I1)    [right=of s1]       {$\bar I_1$};
  \node (s2)    [right=of I1]       {$\subset$};
  \node (I2)    [right=of s2]       {$\bar I_2$};
  \node (s3)    [right=of I2]       {$\subset$};
  \node (dots)  [right=of s3]       {$\cdots$};
  \node (s4)    [right=of dots]     {$\subset$};
  \node (Ir1)   [right=of s4]       {$\bar I_{r-1}$};
  \node (s5)    [right=of Ir1]      {$\subset$};
  \node (Ir)    [right=of s5]       {$~\bar I_r$};

  \node (J0)    [below=1cm of I0]   {$0$};
  \node (t1)    [right=of J0]       {$=$};
  \node (J1)    [right=of t1]       {$J_1$};
  \node (t2)    [right=of J1]       {$\subset$};
  \node (J2)    [right=of t2]       {$J_2$};
  \node (t3)    [right=of J2]       {$\subset$};
  \node (dots2) [right=of t3]       {$\cdots$};
  \node (t4)    [right=of dots2]    {$\subset$};
  \node (Jr1)   [right=of t4]       {$J_{r-1}$};
  \node (t5)    [right=of Jr1]      {$\subset$};
  \node (Jr)    [right=of t5]       {$J_r$};


  \node at ($(I1)!0.5!(J1)$) {$\rotatebox{90}{$=$}$};
  \node at ($(I2)!0.5!(J2)$) {$\rotatebox{90}{$\subset$}$};
  \node at ($(Ir1)!0.5!(Jr1)$) {$\rotatebox{90}{$\subset$}$};
  \node at ($(Ir)!0.5!(Jr)$) {$\rotatebox{90}{$=$}$};

\end{tikzpicture}
\]
We define the {\bf \boldmath{$k^\text{th}$} intermediate Cordovil algebra} $\Cor_k(\cA,\F) := \bar R/J_k$,
and we have surjections
\[
\Cor_k(\cA,\F) = \bar R/J_k \twoheadrightarrow \bar R/\bar I_k \twoheadrightarrow \bar R/\bar I_r = \bar R/J_r = \Cor(\cA,\F).
\]
Theorem \ref{main} has the following corollary.

\begin{corollary}\label{cor:cordovil}
If $J_2 =  J_r$ (that is, if $\Cor(\cA,\F)$ is quadratic), then $\cA$ is clean.
\end{corollary}

\begin{proof}
From the sequence of surjections above, we see that the condition $J_2 = J_r$ implies that $\bar I_2 = \bar I_r$.
Since $\dim \bar I = \dim I$ for any ideal $I\subset R$, this implies that $I_2 = I_r$, which is equivalent to cleanliness
by Theorem \ref{main}.
\end{proof}

The converse to Corollary \ref{cor:cordovil} is false because the inclusion $J_2\subset \bar I_2$ need not be an equality.

\begin{example}\label{ex:d4}
  The $D_4$ arrangement consists of the $12$ hyperplanes in $\R^4$ given by equations $x_i\pm x_j=0$ for $1\leq i<j\leq 4$.  A
    Macaulay2~\cite{M2} calculation easily shows that $I_2=I_4$, so
    $D_4$ is clean.  A similar calculation shows
    that the Cordovil ideal $J_4$ has minimal generators
    in degrees $2$ and $4$.  That is, we have $J_2 =J_3 \subsetneq J_4 = \bar I_4 = \bar I_3 = \bar I_2$.
\end{example}

\section{Connections with other properties of arrangements}\label{sec:relationships}
In this section, we prove that
$$\text{chordal} \;\Longrightarrow\; \text{clean} \;\Longrightarrow\; \text{formal},$$
and then collect known relationships between various properties
of arrangements.

\subsection{Chordality}\label{sec:chordal}
We define a matroid to be {\bf chordal} if, for every circuit $C$ of size at least $4$, there exist circuits $D_1$ and $D_2$ such that $\abs{D_1},\abs{D_2} \geq 3$,  $\abs{D_1\cap D_2} = 1$, and
\[
    C = (D_1\cup D_2)\setminus (D_1\cap D_2)\,.
\]
This definition generalizes the definition of a chordal graph.  We say that $\cA$ is chordal if its associated matroid is chordal.

\begin{remark}
The concept of chordality for graphs goes back to Berge \cite{Berge} and Dirac \cite{Dirac}.
Stanley noticed the connection between chordal graphs and supersolvability \cite[Example 2.7, Proposition 2.8]{Stanley}.
Independently, Barhona and Grötschel introduced the notion of a chordal circuit as a way to characterize the facet-defining hyperplanes of the cycle polytope of a binary matroid \cite[p.53]{Barahona-Groetschel}.
Ziegler then showed that every binary supersolvable matroid not containing the Fano matroid is graphical \cite[Theorem 2.7]{Ziegler}.
Later Cordovil, Forge, and Klein showed that every binary supersolvable matroid is chordal \cite[Theorem 2.2]{Cordovil-Forge-Klein}.
\end{remark}

\begin{theorem}\label{thm:chordal-means-yoshi}
If $\cA$ is chordal, then $\cA$ is clean.
\end{theorem}

\begin{proof}
Let $\cA$ be a chordal arrangement of rank $r$, and consider a sign vector $\epsilon\in\{\pm\}^\cA$ such that $\epsilon\notin\Sigma_r$.
This means that there is a subset $S\subset\cA$ such that $H_S^\epsilon=\emptyset$.  Furthermore, we may take $S$ to be of smallest possible
cardinality with this property.  If $|S| = 3$, then $\epsilon\notin\Sigma_2$, which is what we want to show.  Assume now for the sake of contradiction that $|S|>3$.

By chordality, there exist circuits
$D_1$ and $D_2$ with $D_1 \cap D_2 = \{H\}$ and $S = D_1 \cup D_2 \setminus\{H\}$ for some $H\in \cA$.
Since $D_1$ and $D_2$ are circuits, there exist sign vectors $\epsilon_1,\epsilon_2\in\{\pm\}^\cA$ such that
$$H_{D_1}^{\epsilon_1} = \emptyset = H_{D_2}^{\epsilon_2}.$$
We may assume without loss of generality that $\epsilon$ and $\epsilon_1$ agree on at least one element of $S\cap D_1$
(otherwise, replace $\epsilon_1$ with $-\epsilon_1$).
We may also assume without loss of generality that $(\epsilon_1)_H \neq (\epsilon_2)_H$ (otherwise, replace $\epsilon_2$ with $-\epsilon_2$).
Then the strong elimination property for oriented matroids implies that, for $i\in\{1,2\}$, $\epsilon_{H_i} = (\epsilon_i)_{H_i}$ for any $H_i\in S\cap D_i$.

Choose the unique $i\in\{1,2\}$ such that $(\epsilon_i)_H = \epsilon_H$.  Then $\epsilon$ agrees with $\epsilon_i$ on $S$, so 
$H_{D_i}^{\epsilon_i} = \emptyset$.  But $|D_i|<|S|$, which gives a contradiction.
\end{proof}

\begin{example}\label{ex:x2-arr}
The converse to \Cref{thm:chordal-means-yoshi} is false,
as illustrated by the arrangement $X_2$ of hyperplanes in $\R^3$ given by the following equations:    
    \[
        x_1=0, \quad x_2=0, \quad x_3=0, \quad x_2=x_3, \quad x_1=x_3, \quad x_1= -x_2, \quad x_1+x_2-2x_3=0.
    \]
    We can check with Macaulay2 that $I_2=I_3$, hence Theorem \ref{main} implies that $\cA$ is clean.
    The associated matroid are $20$ circuits, $5$ of which have three elements and $15$ of which have four elements.
    As there are only $10$ pairs of $3$-element circuits, $\cA$ cannot be chordal.
\end{example}

\subsection{Formality}\label{subsec:formality}
For each $H\in\cA$, choose a linear functional $\alpha_H\in V^*$  that is positive on $H^+$
(this choice is unique up to positive scaling).  Let $\F^\cA := \F\{e_H\mid H\in\cA\}$, and consider the 
linear map $\pi\colon\F^\A\to V^*$ defined by putting $\pi(e_H)=\alpha_H$ for all $H\in\cA$.
This induces a dual inclusion of $V$ into $\F^\cA$.  Let $V^\perp := \ker(\pi)\subset\F^\cA$, which
may also be interpreted as the orthogonal complement to $V$ with respect to the dot product.

For each flat $F\subseteq\A$ of the associated matroid, let $\pi_F$ be the restriction of
$\pi$ to the coordinate subspace $\F^F\subset\F^\cA$, and let $V_F^\perp :=\ker(\pi_F)\subset V^\perp$.  Let
\[
V_{2}^\perp := \sum_{\rk F = 2}V_F^\perp\subseteq V^\perp,
\]
let $V_2\subset\F^\cA$ be the orthogonal complement of $V_2^\perp$, and let $\pi_2:\F^\cA\to V_2^*$ be the projection.  Then we have the following diagram:
\[
\begin{tikzcd}
  0\ar[r] & V^\perp_{2}\ar[r]\ar[hookrightarrow,d] & \F^\A \ar[r,"\pi_{2}"]\ar[d,"="] & V^*_{2}
  \ar[twoheadrightarrow,d] \ar[r] & 0\\
  0\ar[r] & V^\perp \ar[r] & \F^\A\ar[r,"\pi"] & V^*\ar[r] & 0.
\end{tikzcd}
\]
An arrangement is {\bf formal} in the sense of Falk--Randell~\cite{FR86} if $V = V_2$.  This is equivalent to the statement that all linear relations between
the linear functionals $\a_H$ are generated by those involving only three hyperplanes.
Let $\A_{2}$ denote the arrangement in $V_{2}$ defined by the linear functionals $\pi_{2}(H)$ for $H\in\A$; this is called the 
{\bf formal closure} of $\A$.  

For any flat $F$, let $V_F \subset V$ be the intersection of the hyperplanes in $V$, and let $$\cA_F := \{H/V_F\mid D\in F\}$$ 
denote the {\bf localization} of $\cA$ at $V$, which is an essential arrangement in the vector space $V/V_F$.

\begin{proposition}\label{prop:local chambers}
  For any arrangement $k\geq 1$ and any sign vector $\epsilon\in\{\pm\}^\cA$, $\epsilon\in\Sigma_k(\cA)$ if 
  and only if $\epsilon|_F\in\sigma_k(\cA_F)$ for all flats $F$ of rank $k$.
\end{proposition}

\begin{proof}
Suppose $\epsilon\in\Sigma_k(\cA)$ and $F$ is a flat of rank $k$.  By Helly's
theorem, $H^\epsilon_F\neq\emptyset$, which means that $\epsilon|_F\in\sigma_k(\cA_F)$.
Conversely, suppose that $\epsilon|_F\in\sigma_k(\cA_F)$ for all flats $F$ of rank $k$,
let $S\subset\cA$ be a subset of cardinality $k+1$, and let $F$ be the smallest flat containing $S$.
If $S$ is independent, then $H^\epsilon_S\neq\emptyset$.  If $S$ is dependent, then $F$ has rank at most $k$,
and $H^\epsilon_S \supseteq H^\epsilon_F \neq \emptyset$, so $\epsilon\in \Sigma_k(\cA)$.
\end{proof}

For lack of a reference, we state and prove the following elementary lemma. 

\begin{lemma}\label{lem:chambers}
  Suppose $\cA$ is an essential arrangement in a real vector space $V$, $V'\subsetneq V$ is a linear subspace that is not contained in any element of $\cA$, and
  \[
  \cA'=\set{H\cap V'\mid H\in \cA}.
  \]
Then $\abs{\cC(\cA')}<\abs{\cC(\cA)}$.
\end{lemma}

\begin{proof}
  It suffices to assume $V'$ is a hyperplane in $V$.  Choose $\alpha\in V^*$
  so that $V'=\ker \alpha$.  Since $\cA$ is essential,
  it contains a Boolean arrangement $\cB$ of rank $r=\dim V$.  The $1$-dimensional flats of $\cB$ are spanned by basis vectors $v_1,\ldots,v_r$
  for $V$, and we may choose their signs so that $\alpha(v_i)\geq 0$ for each
  $i$.  Then $\alpha$ is strictly positive on the cone $\R_{>0}\set{v_1,\ldots,v_r}$, which is a chamber of $\cB$.
We have natural maps
\[
  \cC(\cA')\hookrightarrow \cC(\cA)\twoheadrightarrow
  \cC(\cB).
  \]
  We showed the composite is not surjective, so neither is the first map.
  \end{proof}

The following result is due to Paul M\"ucksch~\cite{Mue25}.
\begin{theorem}\label{thm:pm}
  If $\A$ is clean, then $\A$ is formal.
\end{theorem}

\begin{proof}
Suppose $\A$ is not formal.  Then $V_2\supsetneq V$, so Lemma~\ref{lem:chambers} tells us that 
$\cA_2$ has more chambers than $\cA$.  Since chambers of $\cA$ are in bijection with $\Sigma_{\rk \cA}(\cA)$,
this means that there exists a sign vector $\epsilon\in\Sigma_{\rk \cA_2}(\cA_2)\setminus\Sigma_{\rk \cA}(\cA)$.
For every flat $F$ of rank 2, we have $$\epsilon|_F \in \Sigma_{2}((\cA_2)_F) = \Sigma_{2}(\cA_F),$$
so $\epsilon\in\Sigma_2(\cA)$ by Proposition~\ref{prop:local chambers}.
But $\epsilon\not\in\Sigma_{\rk \cA}(\A)$, so $\A$ is not clean.
\end{proof}
Following \cite[Def.\ 2.3]{Yuz93}, we say an arrangement $\A$ is {\bf locally
  formal} if the localization $\A_F$ is formal for all $F$.  Taking
$F=\A$, we note that locally formal arrangements are formal.
\begin{corollary}\label{cor:yoshi-typeA}
  If $\A$ is a graphical arrangement, the following are equivalent:
  \begin{enumerate}
  \item[(1)] $\cA$ is chordal
  \item[(2)] $\A$ is clean\label{eq:cor1}
    \item[(3)] $\A$ is locally formal\label{eq:cor2}
    \item[(4)] $\A$ is $K(\pi,1)$.  \label{eq:cor3}
\end{enumerate}
\end{corollary}

\begin{proof}
  Theorem \ref{thm:chordal-means-yoshi} tells us that (1) implies (2).
  A fortiori, if $\A$ is clean, then so is every localization $\A_F$.  Using
  Theorem \ref{thm:pm} for each $\A_F$ shows that (2) implies (3).

  Suppose $\A$ is not chordal: then it has a localization $\A_F$ to a graph
  which is a circuit of length at least $4$.  Such an arrangement is not
  formal, so $\A$ is not locally formal, and (3) implies (1).

Chordal graphical arrangements are supersolvable,
hence $K(\pi,1)$ \cite{FR87}, so (1) implies (4).
Finally, (4) implies (2) by Theorem~\ref{thm:yoshinaga}.
\end{proof}

\begin{remark}\label{formal converse}
  The converse to Theorem \ref{thm:pm} is false.  Ziegler~\cite[Ex.\ 8.7]{Zie89} provided a provided
  a pair of combinatorially equivalent rank-$3$ arrangements, distinguished by  whether or not their (six) triple points lie on a conic, shown in Figure~\ref{fig:Ziegler}.  The Varchenko--Gelfand algebras are isomorphic,
  and a Macaulay2~\cite{M2} computation shows they are not quadratic.
  Yuzvinsky noted that the special arrangement is not formal, while the general one is \cite[Ex.\ 2.2]{Yuz93}.
\end{remark}
\begin{remark}
  Not every formal arrangement $\A$ is locally formal: for example, the
  edge graph of an octahedron gives a formal graphical arrangement, but
  the graph is not chordal: we refer to   Toh\u aneanu~\cite{To07} for a
  complete characterization.

  By way of contrast, if $\A$ is formal, it was recently shown that so is
  the arrangement $\A^F$, for all $F$ (\cite[Thm.\ 1.2]{MMR24}).
\end{remark}
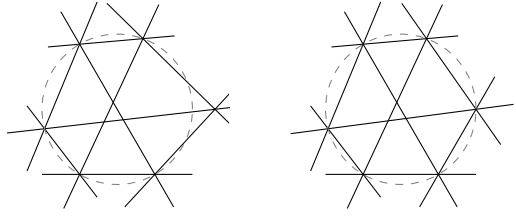
\begin{figure}  
  \[
\begin{tikzpicture}
\begin{scope}
\clip (0,0) circle (1.5);
\coordinate (A) at (0:1.3);
\coordinate (B) at (70:1);
\coordinate (C) at (120:1);
\coordinate (D) at (195:1);
\coordinate (E) at (240:1);
\coordinate (F) at (300:1);
\drawline{(A)}{(B)}
\drawline{(B)}{(C)}
\drawline{(C)}{(D)}
\drawline{(D)}{(E)}
\drawline{(E)}{(F)}
\drawline{(F)}{(A)}
\drawline{(A)}{(D)}
\drawline{(B)}{(E)}
\drawline{(C)}{(F)}
\draw[dashed,gray] (0,0) circle (1);
\end{scope}
\end{tikzpicture}
\qquad
\begin{tikzpicture}
\begin{scope}
\clip (0,0) circle (1.5);
\coordinate (A) at (0:1);
\coordinate (B) at (70:1);
\coordinate (C) at (120:1);
\coordinate (D) at (195:1);
\coordinate (E) at (240:1);
\coordinate (F) at (300:1);
\drawline{(A)}{(B)}
\drawline{(B)}{(C)}
\drawline{(C)}{(D)}
\drawline{(D)}{(E)}
\drawline{(E)}{(F)}
\drawline{(F)}{(A)}
\drawline{(A)}{(D)}
\drawline{(B)}{(E)}
\drawline{(C)}{(F)}
\draw[dashed,gray] (0,0) circle (1);
\end{scope}
\end{tikzpicture}
\]
  \caption{Ziegler's pair (in $\P^2$)}\label{fig:Ziegler}
\end{figure}

\subsection{Relationships}\label{subsec:relationships}
Below are several well-known arrangements, together with a summary of which properties they satisfy.
Here OS refers to the Orlik--Solomon algebra, Cord refers to the Cordovil algebra, and both quad Cord and quad OS mean that the defining ideals of the corresponding rings are quadratically generated.

\begin{center}
\resizebox{\columnwidth}{!}{%
\begin{tabular}{|c|cccccc|}
    \hline
\textbf{Arrangement} & $\mathbf{K(\pi,1)}$ & \textbf{free} & \textbf{formal} & \textbf{clean} & \textbf{quad Cordovil} & \textbf{quad OS}\\
\hline
Falk \cite[Example 3.13]{falk95} & \checkmark & \checkmark & \checkmark & \checkmark & $\times$ & $\times$ \\ 
DY \cite[Example 4.6]{DY02} & ? & $\times$ & \checkmark & \checkmark & $\times$ & $\times$ \\ 
Ziegler$_1$ \cite[Example 8.7]{Zie89} & $\times$ & $\times$ & \checkmark & $\times$ & $\times$ & $\times$ \\ 
Ziegler$_2$ \cite[Example 8.7]{Zie89} & $\times$ & $\times$ & $\times$ & $\times$ & $\times$ & $\times$ \\
ER \cite[Theorem 2.1 ($\alpha=-1$)]{Edelman-Reiner3}  & $\times$ & \checkmark & \checkmark & \checkmark & \checkmark & \checkmark  \\
ER \cite[Theorem 2.1 ($\alpha=0$ or $1$)]{Edelman-Reiner3}  & $\times$ & \checkmark & \checkmark & \checkmark & $\times$ & $\times$  \\
$D_4$ (Example \ref{ex:d4}) & \checkmark & \checkmark & \checkmark & \checkmark & $\times$ & $\times$ \\ 
$X_2$ (Example \ref{ex:x2-arr}) & $\times$ & $\times$ & \checkmark & \checkmark & \checkmark & \checkmark  \\
\hline
\end{tabular}
}
\end{center}

The following diagram summarizes the relationships (and non-relationships) between some of these properties.

\begin{tikzpicture}[
    node distance=2cm,
    every node/.style={rounded corners, align=center},
]
\node[] (yoshi) {Clean};

\node[below of=yoshi,xshift=-0.5cm,yshift=-0.5cm] (quadVG) {Quadratic VG};
\node[above of=yoshi] (chordal) {Chordal};
\node[left of=yoshi, xshift=-3cm] (quadgr) {Quadratic Cordovil};
\node[below of=quadVG,xshift=-0.5cm,yshift=-0.5cm] (quados) {Quadratic OS};
\node[right of=yoshi,xshift=3cm,yshift=2cm] (kpi) {$K(\pi,1)$};
\node[below of=kpi,yshift=-1cm] (formal) {Formal};

\node[below of=formal,xshift=1cm,yshift=-1cm] (supersolvable) {Supersolvable};
\node at ($(formal)!0.5!(supersolvable)$) (free) {Free};
\node[below of=supersolvable,xshift=0.5cm] (koszulos) {Koszul OS};
\node at ($(koszulos)!0.5!(quados)$) (rationalkpi) {Rational $K(\pi,1)$};

\draw[<->] (yoshi) to node[midway, left]{Thm.\ \ref{main}} (quadVG);
\draw[<-,red,dashed, bend left=20] (chordal) to node[midway, right, text=red]{$X_2$} (yoshi);
\draw[->,bend right=20] (chordal) to node[midway,left]{Thm.\,\ref{thm:chordal-means-yoshi}} (yoshi);
\draw[->, red, dashed] (yoshi) -- node[midway, above, text=red]{Falk, DY} (quadgr);
\draw[->, red, dashed] (quadVG) to node[midway, right, text=red]{$D_4$} (quados);
\draw[->] (yoshi) -- (formal) node[midway,above] {Thm.\,\ref{thm:pm}};
\draw[->, bend right=10] (quadgr) to node[midway,below]{Cor.\,\ref{cor:cordovil}} (yoshi);
\draw[->] (quados) -- (formal);
\draw[->, red, dashed,bend left = 15] (formal) to node[midway, below, text=red]{~$D_4$} (quados);
\draw[->,red,dashed, bend left=20] (yoshi) to node[midway, above, text=red]{$X_2$} (kpi);
\draw[<-] (yoshi) to node[midway, right]{Thm.\,\ref{thm:yoshinaga}} (kpi);
\draw[->] (kpi) to (formal);
\draw[->] (rationalkpi) to (quados);
\draw[<->] (rationalkpi) to (koszulos);
\draw[->] (supersolvable) to (koszulos);
\draw[->] (supersolvable) to (free);
\draw[->,bend right=50] (supersolvable) to (kpi);
\draw[->] (free) to (formal);
\draw[<-,red,dashed,bend left = 60] (kpi) to node[midway, right, text=red]{$D_4$} (koszulos);
\draw[<-,red,dashed,bend left = 30] (kpi) to node[midway, right, text=red]{ER} (free);
\draw[->,red,dashed] (formal) to node[midway, above, text=red]{Ziegler$_1$} (quadVG);
\draw[->,bend right=10] (quadVG) to (formal);

\node[draw, align=left, anchor=north east, font=\footnotesize] at ([yshift=-2cm,xshift=-2cm]yoshi.south west) {
  \textbf{Legend:} \\
  \tikz[baseline] \draw[->] (0,0) -- (0.6,0); \quad Tail $\Rightarrow$ Head \\
  \tikz[baseline] \draw[<->] (0,0) -- (0.6,0); \quad Head $\iff$ Tail \\
  \tikz[baseline] \draw[->, red, dashed] (0,0) -- (0.6,0); \quad Tail $\not\Rightarrow$ Head \\
};
\end{tikzpicture}

\bibliographystyle{amsalpha}
\bibliography{bibliography}

\end{document}